\documentclass[11pt]{article}

\usepackage[T1]{fontenc}
\usepackage[utf8]{inputenc}
\usepackage{lmodern}
\usepackage{amsmath,amssymb,amsthm,mathtools}
\usepackage[a4paper,margin=30mm]{geometry}
\usepackage{microtype}
\usepackage{xcolor}
\usepackage[hidelinks]{hyperref}
\hypersetup{
  pdftitle={Reflexivity of Grothendieck Spaces and Strongly Normal Sequences},
  pdfauthor={Wieslaw Sliwa},
  pdfsubject={Banach spaces; Grothendieck spaces; strongly normal sequences},
  pdfkeywords={Grothendieck space, reflexivity, strongly normal sequence, separable quotient, L1 space}
}

\newtheorem{theorem}{Theorem}
\newtheorem{lemma}[theorem]{Lemma}
\newtheorem{corollary}[theorem]{Corollary}
\newtheorem{proposition}[theorem]{Proposition}
\newtheorem{remark}[theorem]{Remark}
\theoremstyle{definition}

\newcommand{\N}{\mathbb N}
\newcommand{\lin}{\operatorname{lin}}
\newcommand{\Step}[1]{\par\medskip\noindent\textit{Step #1}\ }

\title{Reflexivity of Grothendieck Spaces and Strongly Normal Sequences}
\author{Wies{\l}aw {\'S}liwa\\
\small Institute of Mathematics, Faculty of Exact and Technical Sciences, University of Rzesz\'ow\\
\small Pigonia 1, 35-310 Rzesz\'ow, Poland\\
\small \texttt{wsliwa@ur.edu.pl}}
\date{October 2026}

\begin{document}
\maketitle

\begin{abstract}
A normalized sequence $(y_n)$ in the dual $E^*$ of a Banach space $E$ is
called normal if it is weak-star null, and strongly normal if
\[
 G(y_n)=\left\{x\in E:\sum_{n=1}^{\infty}|y_n(x)|<\infty\right\}
\]
is dense in $E$.  We prove that, within the class of Grothendieck Banach
spaces, reflexivity is characterized by the subsequence behavior of normal
sequences: a Grothendieck space $E$ is reflexive if and only if every normal
sequence in $E^*$ contains a strongly normal subsequence.  If $E$ is
nonreflexive and Grothendieck, then $E^*$ contains a normalized weakly null
sequence none of whose subsequences is strongly normal; more strongly, it
contains a closed hereditarily $\ell_1$ subspace $Y$ with a normalized
weakly null Schauder basis such that no sequence in $S(Y)$ is strongly
normal.  We also obtain an exact quotient criterion valid for arbitrary
Banach spaces: a normal sequence has a strongly normal subsequence precisely
when, for some subsequence $z=(z_n)$, the quotient
$E/\bigcap_n\ker z_n$ is separable.  The proofs combine a
weak-star basic-sequence argument with results of Haydon,
Johnson--Maurey--Schechtman and Popov.  Concrete realizations are given in
$(L_\infty(\mu))^*$ for nonatomic probability spaces (and, more generally,
for $\sigma$-finite spaces with a nonatomic part) and in $\ell_\infty^*$.
\end{abstract}



\section{Introduction}
Let $E$ be a real Banach space.  Following \cite{Sliwa1996,Sliwa2012}, a
sequence $(y_n)\subset S(E^*)$ is called \emph{normal} if
\[
 y_n\longrightarrow 0\quad\text{in }\sigma(E^*,E).
\]
It is called \emph{strongly normal} if
\[
 G(y_n):=\left\{x\in E:\sum_{n=1}^{\infty}|y_n(x)|<\infty\right\}
\]
is dense in $E$.  Every strongly normal sequence is normal.  Indeed, if
$x\in E$ and $\varepsilon>0$, choose $g\in G(y_n)$ with
$\|x-g\|<\varepsilon$.  Since $\sum_n |y_n(g)|<\infty$, one has
$y_n(g)\to0$, and, because $\|y_n\|=1$,
\[
 |y_n(x)|\le |y_n(x-g)|+|y_n(g)|
            \le \varepsilon+|y_n(g)|.
\]
Thus $y_n(x)\to0$ for every $x\in E$.

The following question was first posed by the author as Problem~7.11 in his
doctoral dissertation, written and defended in Polish at Adam Mickiewicz
University in 1996 \cite{Sliwa1996}:
\[
 \textit{Does every normal sequence in $E^*$ contain a strongly normal
 subsequence?}
\tag{P}
\]
The same dissertation already proved two facts that will be used below: every
strongly normal sequence has a weak-star basic subsequence
\cite[Proposition~7.7]{Sliwa1996}, and Problem~(P) has a positive answer for
weakly compactly generated (WCG) spaces \cite[Proposition~7.12]{Sliwa1996}.
Before the journal publication, W\'ojtowicz extended the WCG result in 2002
to spaces with the controlled separable projection property, in particular to
weakly compactly determined spaces \cite[Lemma~1]{Wojtowicz2002}.  Problem~(P)
and the original 1996 results were subsequently published by the author in
\cite[Theorem~1 and Proposition~4]{Sliwa2012}.  The problem was still listed
explicitly as an open problem in the 2018 survey \cite[Problem~7]{Ferrando2018}.
A related positive selection result appeared recently in
\cite{KakolLopezSliwa2026}: for infinite compact spaces $X$ and $Y$, the
authors exhibited a concrete normal sequence in $C(X\times Y)^*$ such that
every subsequence admits a further subsequence which is strongly normal and
weak-star basic.

In September 2026 K\k{a}kol posted a ZFC negative solution
\cite{Kakol2026}.  He constructed, for every set $\Gamma$ with
$|\Gamma|\ge\mathfrak c$, a normalized weakly null sequence in
$\ell_1(\Gamma)^*=\ell_\infty(\Gamma)$ having no strongly normal subsequence,
and also obtained a counterexample in a Banach sequence space not isomorphic
to any $\ell_1(\Gamma)$.  His method is combinatorial and, for the family
$\ell_1(\kappa)$, yields additional cardinal information involving the
bounding number $\mathfrak b$; in particular, the corresponding assertion for
$\kappa=\omega_1$ is independent of ZFC.

The present work gives a complementary structural result of a different
kind.  K\k{a}kol's theorem concerns specific classes and, in addition, detects
cardinal thresholds.  Our quantifier is instead structural: we show that the
obstruction occurs in the dual of \emph{every} nonreflexive Grothendieck
space and obtain a characterization of reflexivity inside the Grothendieck
class.  We then strengthen the conclusion beyond subsequences: the
obstruction persists for normalized block sequences of a suitable weakly
null basic sequence, and indeed inside an entire closed hereditarily
$\ell_1$ subspace of the dual there is no normalized strongly normal
sequence at all; see Theorem~\ref{thm:hereditary-no-strong}.  Thus the two
approaches have different scopes and neither subsumes the other.

The proof uses three ingredients.  First, every strongly normal sequence in
$E^*$ contains a subsequence which is a Schauder basic sequence in
$(E^*,\sigma(E^*,E))$ \cite[Proposition~7.7]{Sliwa1996}; see also
\cite[Theorem~1]{Sliwa2012}.  Second, we prove that in
the dual of a Grothendieck space every such weak-star basic sequence has
reflexive norm-closed linear span.  Third, Haydon proved that the dual of every
nonreflexive Grothendieck space contains an isometric copy of a large
$L_1$-space \cite{Haydon1987}; in particular it contains an isometric copy of
$L_1[0,1]$.  We then insert into this copy the normalized weakly null sequence
constructed by Johnson, Maurey and Schechtman \cite{JMS2007}, such that every
subsequence has nonreflexive closed linear span.

A further aim is to isolate the quotient-theoretic content of the selection
problem.  For a subsequence $z=(z_n)$ put
\[
 K(z)=\bigcap_{n=1}^{\infty}\ker z_n.
\]
We prove that a normal sequence contains a strongly normal subsequence if and
only if some subsequence $z$ detects a separable quotient $E/K(z)$.  In the
Grothendieck setting this is equivalent to the existence of a subsequence
whose norm-closed linear span is reflexive.  This criterion provides the
bridge between the selection problem and the structural obstructions used
later in the paper.

Finally, we make the construction concrete.  First we record a completely
specified dyadic version of the Johnson--Maurey--Schechtman sequence.  We then
show that every nonatomic probability space $(\Omega,\Sigma,\mu)$ admits a
dyadic splitting which transfers this sequence isometrically into $L_1(\mu)$.
Ordinary integration then gives an explicit counterexample in
$(L_\infty(\mu))^*$.  As the simplest special case one recovers
$L_\infty[0,1]$.  We finally give a concrete representation in
$\ell_\infty^*$, where the only noncanonical ingredient is the choice of a
free ultrafilter.

For simplicity we work over the real scalar field, which is also the setting
of the cited $L_1$ constructions.  Throughout, $\mathbb N=\{1,2,\ldots\}$.

\section{Weak-star basic sequences in duals of Grothendieck spaces}
The first lemma is the structural bridge used throughout the paper.  Its proof
identifies the weak-star closed span of a weak-star basic sequence with the
dual of a separable quotient of the underlying space.  In the Grothendieck
setting that quotient is reflexive, and this forces reflexivity of the
norm-closed span of the sequence.
Recall that a Banach space $E$ is a \emph{Grothendieck space} if every
weak-star convergent sequence in $E^*$ is weakly convergent.  As in
\cite{Sliwa2012}, a Schauder basic sequence in $(E^*,\sigma(E^*,E))$ means a
Schauder basis of its $\sigma(E^*,E)$-closed linear span; in particular, its
coefficient functionals are continuous for the induced weak-star topology.

\begin{lemma}\label{lem:wstar-basic}
Let $E$ be a Grothendieck Banach space and let $(z_n)\subset E^*$ be a Schauder
basic sequence in $(E^*,\sigma(E^*,E))$.  Then its norm-closed linear span
\[
 [z_n]:=\overline{\lin\{z_n:n\in\N\}}^{\|\cdot\|}
\]
is reflexive.
\end{lemma}

\begin{proof}
We divide the argument into seven steps.  The purpose is to identify the
weak-star closed span of $(z_n)$ with the dual of a separable quotient of
$E$, and then to use the Grothendieck property to make that quotient
reflexive.

\Step{1: pass from the weak-star closed span to a quotient.}
Let
\[
 Y=\overline{\lin\{z_n:n\in\N\}}^{\sigma(E^*,E)}
\]
and set
\[
 M={}^\perp Y=\{x\in E:y(x)=0\text{ for every }y\in Y\}.
\]
The subspace $M$ is closed because it is an intersection of kernels of
continuous linear functionals.  Let $q:E\to E/M$ be the quotient map.  Since
$Y$ is weak-star closed and $M={}^\perp Y$, the annihilator theorem gives
\[
 Y=M^\perp=q^*((E/M)^*).
\tag{1}
\]
Thus from now on we may regard $Y$ as $(E/M)^*$.

\Step{2: represent the coefficient functionals by vectors of $E$.}
Let $\varphi_n:Y\to\mathbb R$ be the $n$-th coefficient functional of the
weak-star Schauder basis $(z_n)$.  By definition, $\varphi_n$ is continuous
for the topology induced by $\sigma(E^*,E)$ on $Y$.  By weak-star continuity, there exist finitely many vectors
$u_1,\ldots,u_m\in E$ and $C>0$ such that
\[
 |\varphi_n(y)|\le C\max_{1\le j\le m}|y(u_j)|,
 \qquad y\in Y.
\]
Define
\[
 T:Y\to\mathbb R^m,\qquad T(y)=(y(u_1),\ldots,y(u_m)).
\]
If $T(y)=0$, the displayed estimate gives $\varphi_n(y)=0$.  Therefore
$\varphi_n$ factors through the subspace $T(Y)\subset\mathbb R^m$.  Extend the
resulting functional on $T(Y)$ to all of $\mathbb R^m$.  Hence for suitable
scalars $a_1,\ldots,a_m$,
\[
 \varphi_n(y)=\sum_{j=1}^m a_jy(u_j)
             =y\!\left(\sum_{j=1}^m a_ju_j\right).
\]
With
\[
        x_n=\sum_{j=1}^m a_ju_j\in E
\]
we obtain
\[
        \varphi_n(y)=y(x_n),\qquad y\in Y.
\tag{2}
\]

\Step{3: show that the quotient $E/M$ is separable.}
We claim that
\[
 \overline{\lin\{q(x_n):n\in\N\}}=E/M.
\tag{3}
\]
Suppose not.  Let
$L=\overline{\lin\{q(x_n):n\in\N\}}$, a proper closed subspace of $E/M$.
By Hahn--Banach there is a nonzero functional
$\eta\in(E/M)^*$ such that $\eta|_L=0$.  In particular,
$\eta(q(x_n))=0$ for every $n$.

By (1), $q^*\eta\in Y$.  Using (2),
\[
 \varphi_n(q^*\eta)
 =(q^*\eta)(x_n)
 =\eta(q(x_n))=0
 \qquad(n\in\mathbb N).
\]
Thus all coefficients of the vector $q^*\eta\in Y$ with respect to the
Schauder basis $(z_n)$ are zero.  Uniqueness of basis expansions gives
$q^*\eta=0$.  Since $q^*$ is injective, $\eta=0$, contradicting the choice of
$\eta$.  Therefore (3) holds.

The set $\{q(x_n):n\in\mathbb N\}$ is countable and has dense linear span, so
$E/M$ is separable.

\Step{4: prove that $E/M$ is Grothendieck.}
Take a weak-star null sequence $(\eta_j)\subset(E/M)^*$.  For every $x\in E$,
\[
 (q^*\eta_j)(x)=\eta_j(qx)\longrightarrow0,
\]
so $(q^*\eta_j)$ is weak-star null in $E^*$.  Because $E$ is Grothendieck,
\[
        q^*\eta_j\longrightarrow0
        \quad\text{weakly in }E^*.
\]
We need this convergence inside the subspace $M^\perp=q^*((E/M)^*)$.
The weak topology of a norm-closed subspace is exactly the topology inherited
from the ambient weak topology: every bounded linear functional on $M^\perp$
extends by Hahn--Banach to a bounded linear functional on $E^*$, that is, to
an element of $E^{**}$.  Therefore
$q^*\eta_j\to0$ weakly in $M^\perp$.

Finally $q^*:(E/M)^*\to M^\perp$ is an isometric isomorphism, so pulling the
weak convergence back through $q^*$ gives
\[
        \eta_j\longrightarrow0
        \quad\text{weakly in }(E/M)^*.
\]
Hence $E/M$ is Grothendieck.

\Step{5: use separability plus Grothendieck to obtain reflexivity.}
We have proved that $E/M$ is both separable and Grothendieck.  Indeed, the dual ball
$B_{(E/M)^*}$ is weak-star compact and metrizable.  Therefore every sequence
in it has a weak-star convergent subsequence.  Grothendieck turns that
subsequence into a weakly convergent one.  Thus the dual ball is weakly
sequentially compact.  By Eberlein--Smulian it is weakly compact.  Hence
$(E/M)^*$, and therefore $E/M$, is reflexive.

\Step{6: identify the weak-star and weak topologies on $Y$.}
By (1), $Y$ is isometrically $(E/M)^*$.  Evaluation of $y\in Y$ at a vector
$x\in E$ depends only on the coset $qx$, because every $y\in Y=M^\perp$
vanishes on $M$.  Consequently the topology induced on $Y$ by
$\sigma(E^*,E)$ is exactly
\[
        \sigma((E/M)^*,E/M).
\]
Since $E/M$ is reflexive, its canonical image in $(E/M)^{**}$ is all of
$(E/M)^{**}$.  Therefore
\[
        \sigma((E/M)^*,E/M)
        =\sigma((E/M)^*,(E/M)^{**}),
\]
which is the weak topology of $Y$.

\Step{7: identify $Y$ with the norm-closed span of $(z_n)$.}
By definition,
\[
  Y=\overline{\lin\{z_n:n\in\mathbb N\}}^{\,w^*}.
\]
On $Y$ the weak-star topology just identified is the weak topology, hence
\[
  Y=\overline{\lin\{z_n:n\in\mathbb N\}}^{\,w}.
\]
For a linear subspace of a Banach space, weak closure and norm closure
coincide.  One inclusion is immediate because norm convergence implies weak
convergence.  For the reverse inclusion, if a point is outside the norm
closure, Hahn--Banach separates it from that closed subspace, so it is also
outside the weak closure.  Therefore
\[
  Y=\overline{\lin\{z_n:n\in\mathbb N\}}^{\|\cdot\|}=[z_n].
\]
But $Y\cong(E/M)^*$ is reflexive.  Hence $[z_n]$ is reflexive.
\end{proof}

\begin{corollary}\label{cor:strong-reflexive}
Let $E$ be a Grothendieck Banach space.  Every strongly normal sequence in
$E^*$ contains a subsequence which is a Schauder basic sequence in
$(E^*,\sigma(E^*,E))$ and whose norm-closed linear span is reflexive.
\end{corollary}

\begin{proof}
By \cite[Theorem~1]{Sliwa2012}, every strongly normal sequence in $E^*$
has a subsequence which is Schauder basic in $(E^*,\sigma(E^*,E))$.  Apply
Lemma~\ref{lem:wstar-basic}.
\end{proof}

\section{The obstruction in nonreflexive Grothendieck spaces}
We first recall the two external results used in the construction.

\begin{theorem}[Johnson--Maurey--Schechtman \cite{JMS2007}]\label{thm:JMS}
There exists a normalized weakly null sequence $(u_n)$ in $L_1[0,1]$ such that,
for every $\varepsilon>0$, the Haar basis of $L_1[0,1]$ is
$(1+\varepsilon)$-equivalent to a block basis of every subsequence of
$(u_n)$.
\end{theorem}

Consequently, for every subsequence $(u_{n_k})$ the space $[u_{n_k}]$ contains
a closed subspace isomorphic to $L_1[0,1]$.  In particular,
\[
 [u_{n_k}]\quad\text{is nonreflexive for every subsequence }(u_{n_k}).
\tag{4}
\]

We shall also use the following consequence of Haydon's theorem.

\begin{theorem}[Haydon \cite{Haydon1987}]\label{thm:Haydon}
If $E$ is a nonreflexive Grothendieck Banach space, then Haydon's theorem
implies that $E^*$ contains an isometric copy of
\[
 L_1(\{0,1\}^{\omega_1})
\]
with respect to the usual product probability measure.
\end{theorem}

For completeness we record the elementary consequence that is actually needed.

\begin{lemma}\label{lem:L1insideHaydon}
The space $L_1(\{0,1\}^{\omega_1})$ contains an isometric copy of $L_1[0,1]$.
Consequently, if $E$ is a nonreflexive Grothendieck space, then there is a
linear isometry
\[
 J:L_1[0,1]\longrightarrow E^*.
\]
\end{lemma}

\begin{proof}
Choose distinct coordinates $\alpha_1,\alpha_2,\ldots<\omega_1$ and define
\[
 \theta(\varepsilon)=\sum_{n=1}^{\infty}
 \frac{\varepsilon(\alpha_n)}{2^n},
 \qquad \varepsilon\in\{0,1\}^{\omega_1}.
\]
Under the product Bernoulli measure, $\theta$ has Lebesgue distribution on
$[0,1]$; equivalently, $\theta_\#\mu=\lambda$.  Hence
\[
 V:L_1[0,1]\longrightarrow L_1(\{0,1\}^{\omega_1}),
 \qquad Vf=f\circ\theta,
\]
is a well-defined linear isometry, because
\[
 \|Vf\|_1=\int |f\circ\theta|\,d\mu
          =\int_0^1|f|\,d\lambda=\|f\|_1.
\]
Combining $V$ with the isometric embedding supplied by
Theorem~\ref{thm:Haydon} gives the required $J$.
\end{proof}

We can now state the main negative result.

\begin{theorem}\label{thm:main-negative}
Let $E$ be a nonreflexive Grothendieck Banach space.  Then there exists a
normal sequence $(y_n)$ in $E^*$ which contains no strongly normal subsequence.
In fact, $(y_n)$ may be chosen normalized and weakly null in $E^*$.
\end{theorem}

\begin{proof}
Let $(u_n)$ be the normalized weakly null sequence supplied by
Theorem~\ref{thm:JMS}.  Let
\[
        J:L_1[0,1]\to E^*
\]
be the isometric embedding from Lemma~\ref{lem:L1insideHaydon}, and set
\[
        y_n=Ju_n.
\]
We verify the required properties one at a time.

\Step{1: normalization.}
Because $J$ is an isometry,
\[
        \|y_n\|=\|Ju_n\|=\|u_n\|_1=1.
\]

\Step{2: weak nullity.}
Take any $F\in E^{**}$.  Then $J^*F\in (L_1[0,1])^*=L_\infty[0,1]$, and
\[
        F(y_n)=F(Ju_n)=(J^*F)(u_n)\to0
\]
because $u_n\to0$ weakly in $L_1[0,1]$.  Hence $y_n\to0$ weakly in $E^*$.
Weak convergence implies weak-star convergence, so $(y_n)$ is normal.

\Step{3: assume a strongly normal subsequence exists.}
Suppose, for contradiction, that $(y_{n_k})$ is strongly normal.  Since $E$
is Grothendieck, Corollary~\ref{cor:strong-reflexive} gives a further
subsequence $(y_{n_{k_j}})$ such that
\[
        [y_{n_{k_j}}]
\]
is reflexive.

\Step{4: identify this closed span with a JMS closed span.}
Because $J$ is an isomorphic embedding with closed range,
\[
        [y_{n_{k_j}}]
        =[Ju_{n_{k_j}}]
        =J([u_{n_{k_j}}]).
\]
The last equality deserves a comment.  Finite linear combinations are
carried by $J$ to the corresponding finite linear combinations; because $J$
is an isometry, norm limits are preserved in both directions on its range.

\Step{5: contradict reflexivity.}
The JMS property says that every subsequence $(u_{n_{k_j}})$ has closed span
containing a closed subspace isomorphic to $L_1[0,1]$.  Hence
$[u_{n_{k_j}}]$ is nonreflexive.  Isomorphic Banach spaces have the same
reflexivity status, so its image $J([u_{n_{k_j}}])=[y_{n_{k_j}}]$ is also
nonreflexive.  This contradicts Step~3.

Therefore no subsequence of $(y_n)$ can be strongly normal.
\end{proof}

\section{A characterization of reflexivity}
The preceding theorem and the positive WCG result give a complete
characterization inside the class of Grothendieck spaces.

\begin{theorem}\label{thm:characterization}
Let $E$ be a Grothendieck Banach space.  The following conditions are
equivalent:
\begin{enumerate}
\item $E$ is reflexive;
\item every normal sequence in $E^*$ contains a strongly normal subsequence.
\end{enumerate}
\end{theorem}

\begin{proof}
Assume first that $E$ is reflexive.  Then $E$ is WCG.  If $E$ is
infinite-dimensional, \cite[Proposition~4]{Sliwa2012} shows that every
normal sequence in $E^*$
contains a strongly normal subsequence.  If $E$ is
finite-dimensional, there is no normal sequence in $S(E^*)$, so (2) holds
vacuously.

Conversely, suppose that $E$ is not reflexive.  By
Theorem~\ref{thm:main-negative}, $E^*$ contains a normal sequence with no
strongly normal subsequence, contradicting (2).  Hence $E$ is reflexive.
\end{proof}

\begin{corollary}\label{cor:linfty}
There exists a normalized weakly null sequence in $\ell_\infty^*$ which has no
strongly normal subsequence.
\end{corollary}

\begin{proof}
The space $\ell_\infty$ is a nonreflexive Grothendieck space, so the assertion
follows from Theorem~\ref{thm:main-negative}.
\end{proof}

\section{Criteria for a normal sequence to contain a strongly normal subsequence}\label{sec:criteria}
The next result isolates the exact local obstruction.  For a subsequence
$z=(z_n)$, the intersection of kernels records the part of $E$ invisible to
$z$, while the quotient $E/K(z)$ is precisely the portion detected by the
sequence.  For an arbitrary Banach space, strong normality after passage to a
subsequence turns out to be equivalent to separability of such a detected
quotient.

For a sequence $z=(z_n)\subset E^*$ put
\[
 K(z)=\bigcap_{n=1}^{\infty}\ker z_n
\]
and
\[
 Y(z)=\overline{\operatorname{lin}\{z_n:n\in\mathbb N\}}
       ^{\,\sigma(E^*,E)}.
\]
Notice that $K(z)={}^\perp Y(z)$ and hence, by the annihilator theorem,
\[
 Y(z)=K(z)^\perp\cong (E/K(z))^*.
\tag{5}
\]
Thus $E/K(z)$ is precisely the quotient of $E$ which is detected by the
functionals $z_n$.

\begin{theorem}\label{thm:criteria-general}
Let $E$ be a Banach space and let $(y_n)\subset S(E^*)$ be a normal sequence.
The following conditions are equivalent.
\begin{enumerate}
\item $(y_n)$ contains a strongly normal subsequence;
\item $(y_n)$ contains a subsequence which is a Schauder basic sequence in
      $(E^*,\sigma(E^*,E))$;
\item $(y_n)$ contains a subsequence $z=(z_n)$ for which the quotient
      $E/K(z)$ is separable;
\item $(y_n)$ contains a subsequence $z=(z_n)$ for which the unit ball of
      $Y(z)$ is metrizable in the topology induced by $\sigma(E^*,E)$.
\end{enumerate}
\end{theorem}

\begin{proof}
We prove the implications in the order
\[
        (1)\Longleftrightarrow(2),\qquad
        (2)\Longrightarrow(3)\Longrightarrow(1),\qquad
        (3)\Longleftrightarrow(4).
\]

\Step{1: equivalence of strong normality and weak-star basic selection.}
The equivalence of (1) and (2) is exactly the combination of two results from
\cite{Sliwa2012}.  Theorem~1 there says that every strongly normal sequence
has a weak-star Schauder basic subsequence.  Conversely, the first assertion
of Proposition~2 there says that every normalized weak-star Schauder basic
sequence is strongly normal.  Since every subsequence of our normal sequence
is still normalized, these results apply without renormalization.

\Step{2: prove (2)$\Rightarrow$(3).}
Assume that a subsequence, relabelled $(z_n)$, is a Schauder basic sequence in
$(E^*,\sigma(E^*,E))$.  Put
\[
        Y=Y(z)
        =\overline{\lin\{z_n:n\in\mathbb N\}}^{\sigma(E^*,E)}
\]
and let $(\varphi_n)$ be its coefficient functionals.  By weak-star continuity, for each $n$ there exists
$x_n\in E$ such that
\[
        \varphi_n(y)=y(x_n),\qquad y\in Y.
\tag{6}
\]
Let $K=K(z)=\bigcap_n\ker z_n$ and $q:E\to E/K$.  Since
$K={}^\perp Y$, the annihilator theorem identifies
\[
        Y=K^\perp=q^*((E/K)^*).
\]
We claim
\[
        \overline{\lin\{q(x_n):n\in\mathbb N\}}=E/K.
\tag{7}
\]
If this failed, Hahn--Banach would give a nonzero
$\eta\in(E/K)^*$ such that $\eta(q(x_n))=0$ for all $n$.  Put
$y=q^*\eta\in Y$.  Then by (6)
\[
        \varphi_n(y)=y(x_n)=\eta(q(x_n))=0
        \qquad(n\in\mathbb N).
\]
All basis coefficients of $y$ vanish, hence $y=0$.  Injectivity of $q^*$ then
gives $\eta=0$, contradiction.  Thus (7) holds.  To spell out the last step, the set of all finite
linear combinations of the vectors $q(x_n)$ with rational coefficients is
countable.  It is dense in the real linear span of the $q(x_n)$, and by (7)
that real linear span is dense in $E/K$.  Therefore $E/K$ has a countable
dense subset and is separable.

\Step{3: prove (3)$\Rightarrow$(1) by a diagonal choice.}
Assume that for a subsequence $z=(z_n)$ the quotient $E/K(z)$ is separable.
Relabel this subsequence as $(z_n)$ and write $K=K(z)$.  Choose vectors
$x_1,x_2,\ldots\in E$ so that
\[
        \{q(x_j):j\in\mathbb N\}
\]
is dense in $E/K$.

For each fixed $j$, normality gives $z_n(x_j)\to0$.  We choose recursively
$n_1<n_2<\cdots$ so that
\[
        |z_{n_k}(x_j)|\le2^{-k}
        \qquad(1\le j\le k).
\tag{8}
\]
Here is the recursion explicitly.  Choose $n_1$ so large that
$|z_{n_1}(x_1)|\le2^{-1}$.  Having chosen $n_1<\cdots<n_{k-1}$, the finitely
many scalar sequences $z_n(x_1),\ldots,z_n(x_k)$ all tend to zero.  Therefore
there exists $n_k>n_{k-1}$ for which all $k$ inequalities
$|z_{n_k}(x_j)|\le2^{-k}$ hold simultaneously.

Fix $j$.  The first $j-1$ terms of
$\sum_k|z_{n_k}(x_j)|$ form a finite number.  For $k\ge j$, (8) applies, so
\[
 \sum_{k=j}^\infty |z_{n_k}(x_j)|
 \le\sum_{k=j}^\infty2^{-k}<\infty.
\]
Hence every $x_j$ belongs to $G(z_{n_k})$.  Also every element of $K$ belongs
to $G(z_{n_k})$, because every $z_n$ vanishes identically on $K$.  Since
$G(z_{n_k})$ is a linear subspace,
\[
 K+\lin\{x_j:j\in\mathbb N\}\subset G(z_{n_k}).
\tag{9}
\]

It remains to check density of the left-hand side.  Let $x\in E$ and
$\varepsilon>0$.  Since the cosets $q(x_j)$ have dense linear span in $E/K$,
there exists $v\in\lin\{x_j\}$ such that
\[
        \|q(x-v)\|<\varepsilon.
\]
By the quotient norm formula
\[
 \|q(x-v)\|=\inf_{k\in K}\|x-v-k\|,
\]
there is $k\in K$ such that
\[
        \|x-v-k\|<2\varepsilon.
\]
(The factor $2$ is inessential: since the infimum is $<\varepsilon$, choose an
approximating $k$ within an additional error $<\varepsilon$.)  Therefore
$K+\lin\{x_j\}$ is dense in $E$.  By (9), $G(z_{n_k})$ contains a dense
subspace, hence is dense.  Thus $(z_{n_k})$ is strongly normal.

\Step{4: prove (3)$\Longleftrightarrow$(4).}
Under the canonical isometric identification
\[
        Y(z)=K(z)^\perp\cong(E/K(z))^*,
\]
the topology induced by $\sigma(E^*,E)$ on $Y(z)$ is exactly the canonical
weak-star topology
\[
        \sigma((E/K(z))^*,E/K(z)).
\]
Indeed, evaluation on $x\in E$ depends only on the coset $q(x)$ because every
functional in $K(z)^\perp$ vanishes on $K(z)$.

We use the standard metrizability criterion: the closed unit ball of $X^*$ is weak-star
metrizable if and only if $X$ is separable.  Apply this with $X=E/K(z)$.
Hence the ball of $Y(z)$ is metrizable in the induced weak-star
topology exactly when $E/K(z)$ is separable.  This proves (3)$\Leftrightarrow$(4).
\end{proof}

Equivalently, a normal sequence has no strongly normal subsequence exactly
when $E/K(z)$ is nonseparable for every subsequence $z$ of that sequence.
This is the negative form of the quotient criterion and does not require the
Grothendieck property.

The quotient condition is often easy to verify directly.  Another useful
sufficient condition is reflexivity of the closed span of a subsequence.
This implication does not require the Grothendieck property.

\begin{proposition}\label{prop:reflexive-span-sufficient}
Let $E$ be a Banach space and let $z=(z_n)\subset E^*$ be any sequence.  If
$Z=[z_n]$ is reflexive, then the quotient $E/K(z)$ is separable.  Consequently,
if $z$ is a subsequence of a normal sequence $(y_n)\subset S(E^*)$, then
$(y_n)$ has a strongly normal subsequence.
\end{proposition}

\begin{proof}
Let $Z=[z_n]$.  The first four steps establish the quotient assertion; no
normality assumption is used there.

\Step{1: $Z$ is weak-star closed in $E^*$.}
Because $Z$ is reflexive, $rB_Z$ is weakly compact for every $r>0$.  The
inclusion map
\[
        (Z,\sigma(Z,Z^*))\longrightarrow(E^*,\sigma(E^*,E))
\]
is continuous: every weak-star evaluation $z\mapsto z(x)$ is a continuous
linear functional on $Z$.  Hence $rB_Z$ is weak-star compact in $E^*$ and
therefore weak-star closed.  Since the norm on $Z$ is inherited from $E^*$,
\[
        Z\cap rB_{E^*}=rB_Z\qquad(r>0).
\]
The Krein--\v{S}mulian theorem now applies to the convex set $Z$: a convex
subset of a dual Banach space is weak-star closed whenever all of its
intersections with closed norm balls are weak-star closed.  Thus $Z$ is
weak-star closed.

\Step{2: represent $Z$ as the dual of a quotient.}
Let $K={}^\perp Z$.  By weak-star closedness and the annihilator theorem,
\[
        Z=K^\perp\cong(E/K)^*.
\tag{10}
\]
Since $Z$ is reflexive, $(E/K)^*$ is reflexive, and therefore $E/K$ is
reflexive.

\Step{3: prove that $E/K$ is separable.}
The space $Z$ is separable because it is the norm-closed span of a sequence.
We recall the standard fact that a separable reflexive Banach space has
separable dual, and include the short argument.  Since $Z$ is separable,
$B_{Z^*}$ is weak-star compact and metrizable; because $Z$ is reflexive, this
weak-star topology is the weak topology of $Z^*$.  Choose a countable weakly
dense set $D\subset B_{Z^*}$.  Then the weak closure of the convex hull of
$D$ is all of $B_{Z^*}$.  By Hahn--Banach, the weak and norm closures of a
convex subset of a Banach space coincide.  Hence the norm-closed convex hull
of $D$ is $B_{Z^*}$.  Rational convex combinations of points of $D$ form a
countable norm-dense subset of this convex hull, so $B_{Z^*}$, and therefore
$Z^*$, is norm separable.  Since $E/K$ is reflexive and
$(E/K)^*\cong Z$, the canonical embedding gives
\[
        E/K\cong (E/K)^{**}\cong Z^*.
\]
Hence $E/K$ is separable.

\Step{4: identify $K$ with $K(z)$.}
Because $Z$ is the norm-closed span of the $z_n$,
\[
 x\in{}^\perp Z
 \quad\Longleftrightarrow\quad
 z_n(x)=0\ \text{for every }n.
\]
Thus
\[
        K={}^\perp Z=\bigcap_n\ker z_n=K(z).
\]
We have therefore shown that $E/K(z)$ is separable.  If $z$ is a subsequence
of a normal sequence $(y_n)$, condition~(3) of
Theorem~\ref{thm:criteria-general} applies and gives a strongly normal
subsequence of $(y_n)$.
\end{proof}

For Grothendieck spaces the preceding sufficient condition becomes an exact
quotient--span equivalence.

\begin{theorem}\label{thm:criteria-groth}
Let $E$ be a Grothendieck Banach space and let $z=(z_n)\subset E^*$ be any
sequence.  Then
\[
        E/K(z)\text{ is separable}
        \quad\Longleftrightarrow\quad
        [z_n]\text{ is reflexive}.
\]
Consequently, if $(y_n)\subset S(E^*)$ is normal, the four equivalent
conditions of Theorem~\ref{thm:criteria-general} are also equivalent to
\begin{enumerate}
\setcounter{enumi}{4}
\item $(y_n)$ contains a subsequence $z=(z_n)$ whose norm-closed linear span
      $[z_n]$ is reflexive.
\end{enumerate}
\end{theorem}

\begin{proof}
Assume first that $E/K(z)$ is separable.  Since quotients of Grothendieck
spaces are Grothendieck (as proved in Step~4 of Lemma~\ref{lem:wstar-basic}),
$E/K(z)$ is a separable Grothendieck space and hence reflexive.  Therefore
\[
        K(z)^\perp\cong(E/K(z))^*
\]
is reflexive.  The space $[z_n]$ is a closed subspace of $K(z)^\perp$, so it
is reflexive as well.

Conversely, if $[z_n]$ is reflexive, Proposition~\ref{prop:reflexive-span-sufficient}
shows directly that $E/K(z)$ is separable.  This proves the displayed
equivalence.  The final assertion now follows by combining it with
condition~(3) of Theorem~\ref{thm:criteria-general}.
\end{proof}

\begin{corollary}\label{cor:criteria-groth-obstruction}
Let $E$ be a Grothendieck Banach space and let $(y_n)\subset S(E^*)$ be
normal.  Then $(y_n)$ has no strongly normal subsequence if and only if
\[
 [y_{n_k}]\quad\text{is nonreflexive for every subsequence }(y_{n_k}).
\tag{11}
\]
\end{corollary}

\begin{proof}
This is the negation of condition~(5) in
Theorem~\ref{thm:criteria-groth}, together with the equivalence of
conditions~(1) and~(5).
\end{proof}

In particular, Corollary~\ref{cor:criteria-groth-obstruction} identifies the
precise obstruction exhibited by the Johnson--Maurey--Schechtman sequence in
our construction: every subsequence has nonreflexive closed linear span.  The
Grothendieck property alone turns every normal sequence into a weakly null
sequence in $E^*$, but weak nullity by itself is not sufficient.  What matters
is whether some subsequence eventually lives in a reflexive closed span, or,
equivalently, whether it detects only a separable quotient of $E$.

\begin{corollary}\label{cor:lp-sufficient}
Let $E$ be a Banach space and let $(y_n)\subset S(E^*)$ be normal.  If some
subsequence of $(y_n)$ is equivalent to the unit vector basis of $\ell_p$ for
some $1<p<\infty$, then $(y_n)$ has a strongly normal subsequence.
More generally, the same conclusion holds whenever some subsequence has
reflexive norm-closed linear span.
\end{corollary}

\begin{proof}
The closed linear span of a sequence equivalent to the unit vector basis of
$\ell_p$, $1<p<\infty$, is isomorphic to the reflexive space $\ell_p$.
Apply Proposition~\ref{prop:reflexive-span-sufficient}.
\end{proof}

Thus, for an individual normal sequence, Theorem~\ref{thm:criteria-general}
reduces the subsequence problem to separability of a quotient determined by
that sequence.  Inside the class of Grothendieck spaces,
Theorem~\ref{thm:criteria-groth} sharpens this further to a purely geometric
criterion: the existence of a strongly normal subsequence is equivalent to
the existence of a subsequence with reflexive closed linear span.

\section{A concrete Johnson--Maurey--Schechtman sequence}\label{sec:concreteJMS}

This section does not introduce a new Johnson--Maurey--Schechtman type
construction.  Its purpose is to make explicit a choice of the parameters and
random variables left flexible in the proof of \cite[Theorem~1]{JMS2007}, so
that the realizations in the next two sections can be written by concrete
formulas.  For those realizations it is convenient to record the resulting
sequence in detail.  Let $\mathcal A$ be the algebra generated by the dyadic
subintervals of $(0,1)$, taken half-open in the usual way; endpoint choices are
irrelevant modulo null sets.  Enumerate its nonempty members as $(D_m)_{m\ge1}$ by increasing
smallest dyadic level and, within each level, lexicographically by incidence
vectors, omitting repetitions and taking $D_1=(0,1)$.  For $n\ge1$ let
$\nu_2(n)$ be the exponent of $2$ in $n$ and set
\[
 E_n=D_{\nu_2(n)+1}.
\]
Then every nonempty member of $\mathcal A$ occurs infinitely often in
$(E_n)$.  We construct powers of two $k_n$, coefficients $a_n=k_n^{3/4}$,
and three-valued dyadic functions $h_{i,n}$.

Set $k_1=16$ and $a_1=8$.  Suppose that $k_r,a_r$ have been chosen for
$1\le r\le N$ and put
\[
 M_N=\sum_{r=1}^{N}a_r
       \left(\frac{\lambda(E_r)}{k_r}\right)^{1/2}.
\tag{12}
\]
Choose $k_{N+1}$ to be the least power of two larger than $k_N$ for which,
with $a_{N+1}=k_{N+1}^{3/4}$,
\[
 2^N k_{N+1}^{1/2}M_N
 \le a_{N+1}\lambda(E_{N+1})
\tag{13}
\]
and
\[
 \frac{a_{N+1}\lambda(E_{N+1})}{k_{N+1}}
 \le
 \min\left\{
 \frac{\min_{1\le r\le N}\lambda(E_r)}{2^{N+1}k_N},
 2^{-N-4}
 \right\}.
\tag{14}
\]
Such a power exists because the right-hand side of (13) grows like
$k_{N+1}^{3/4}$ whereas its left-hand side grows like $k_{N+1}^{1/2}$, and
the left-hand side of (14) is
$\lambda(E_{N+1})k_{N+1}^{-1/4}$.  Condition (13) is the analogue of inequality (7) in the proof of
\cite[Theorem~1]{JMS2007}, while the first inequality in (14) is the
analogue of inequality (8) there.  The second inequality in (14) is an
additional summability condition.  The recursion is explicit in the sense
that every parameter is uniquely specified by the preceding data; it is not
intended as a numerically efficient construction, and the numbers $k_n$ may
grow extremely rapidly.

We next define the functions $h_{i,n}$ without any probabilistic choice.
Let $r_n$ be the least integer for which membership in $E_n$ is determined
by the first $r_n$ binary digits.  Write $k_n=2^{q_n}$.  For
$t\in[0,1)$ use the binary expansion
$t=.\varepsilon_1(t)\varepsilon_2(t)\cdots$ which is not eventually equal
to $1$ (so dyadic rationals use their terminating expansion); the value at
$t=1$ may be fixed arbitrarily.  Put
$b_{i,n}=r_n+(i-1)(q_n+1)$.  Define $h_{i,n}(t)=0$ for $t\notin E_n$, while
for $t\in E_n$ set
\[
 h_{i,n}(t)=
 \begin{cases}
  1,&\varepsilon_{b_{i,n}+1}=\cdots=\varepsilon_{b_{i,n}+q_n}=0
      \text{ and }\varepsilon_{b_{i,n}+q_n+1}=0,\\
 -1,&\varepsilon_{b_{i,n}+1}=\cdots=\varepsilon_{b_{i,n}+q_n}=0
      \text{ and }\varepsilon_{b_{i,n}+q_n+1}=1,\\
  0,&\text{otherwise}.
 \end{cases}
\tag{15}
\]
For fixed $n$, the functions $(h_{i,n})_{i\ge1}$ are independent with
respect to normalized Lebesgue measure on $E_n$, have mean zero, and satisfy
\[
 \|h_{i,n}\|_1=\frac{\lambda(E_n)}{k_n}.
\tag{16}
\]
For completeness, let us record the estimate that will play the role of
inequality~(6) in \cite{JMS2007}.  If $\sigma\subset\mathbb N$ is finite,
then for each fixed $r$, Cauchy--Schwarz on $[0,1]$ gives
\[
 \left\|\sum_{i\in\sigma}h_{i,r}\right\|_1
 \le
 \left(\int_0^1\Bigl(\sum_{i\in\sigma}h_{i,r}\Bigr)^2\,d\lambda\right)^{1/2}.
\]
The cross terms vanish because the $h_{i,r}$ are independent on $E_r$ and
have mean zero, whereas
$\int h_{i,r}^2\,d\lambda=\lambda(E_r)/k_r$.  Hence
\[
 \left\|\sum_{i\in\sigma}h_{i,r}\right\|_1
 \le
 \left(\frac{|\sigma|\lambda(E_r)}{k_r}\right)^{1/2}.
\]
Multiplying by $a_r$ and summing over $1\le r\le N$ yields
\[
 \sum_{r=1}^{N}
 \left\|\sum_{i\in\sigma}a_rh_{i,r}\right\|_1
 \le |\sigma|^{1/2}M_N.
\tag{17}
\]

The functions $h_{i,n}$ satisfy precisely the structural requirements
(i)--(v) imposed in the proof of \cite[Theorem~1]{JMS2007}: their supports lie
in $E_n$, they have mean zero, their supports have relative measure $1/k_n$,
they are measurable with respect to the dyadic algebra, and, for each fixed
$n$, they are independent on $E_n$.  Estimate~(17) substitutes for
inequality~(6) in that proof, (13) substitutes for (7), and the first
inequality in (14) substitutes for (8).

Let us make explicit the point at which these estimates enter Claim~1 of
\cite{JMS2007}.  In that proof one chooses indices $N_j$ with
$E_{N_j}=F_j$ and finite sets $\sigma_j$ of cardinality
$|\sigma_j|=\varepsilon k_{N_j}$, where $\varepsilon^{-1}$ is a power of two.
Conditions (i)--(v), together with Lemma~1 of \cite{JMS2007}, produce the
sets $A_j,B_j$ which approximate
$\sum_{i\in\sigma_j}h_{i,N_j}$.  The remaining step is to replace this sum by
a linear combination of the $f_i$'s.  Precisely here \cite{JMS2007} invokes
its inequalities (6)--(8).  In our notation, (17), (13), and the first
inequality in (14) give
\[
 \left\|\sum_{i\in\sigma_j}
 \bigl(a_{N_j}^{-1}f_i-h_{i,N_j}\bigr)\right\|_1
 \le
 2^{-N_j+1}\varepsilon^{1/2}\lambda(E_{N_j})
 +\varepsilon 2^{-N_j}\lambda(E_{N_j}),
\]
where for the tail we use, for $r>N_j$,
\[
 \frac{a_r\lambda(E_r)}{k_r}
 \le \frac{\lambda(E_{N_j})}{2^r k_{N_j}}.
\]
Since every nonempty member of $\mathcal A$ occurs infinitely often among the
$E_n$, the indices $N_j$ may be chosen as large as required; and because all
$k_n$ are powers of two and tend to infinity, $\varepsilon k_{N_j}$ is an
integer once $k_{N_j}\ge\varepsilon^{-1}$.  Thus the error
estimate used in Claim~1 is reproduced exactly, and the rest of that proof,
as well as the subsequent small-perturbation construction, goes through
unchanged.  The extra bound $2^{-N-4}$ in (14) only forces additional growth
of $k_{N+1}$; it does not weaken any of the hypotheses used there and will be
used below only to make the defining series uniformly summable.

Define
\[
 f_i=\sum_{n=1}^{\infty}a_nh_{i,n}.
\tag{18}
\]
The series converges in $L_1$, uniformly in $i$, since (14) gives
\[
 \sum_{n=2}^{\infty}a_n\|h_{i,n}\|_1
 \le\sum_{n=2}^{\infty}2^{-n-3}=\frac1{16}.
\]
Also $a_1\|h_{i,1}\|_1=1/2$, and hence
\[
 \frac7{16}\le\|f_i\|_1\le\frac9{16},\qquad i\in\mathbb N.
\tag{19}
\]
For each fixed $n$, $h_{i,n}\to0$ weakly in $L_1$ as $i\to\infty$.
Indeed, let $\mathcal F_m$ denote the $\sigma$-algebra generated by the first
$m$ binary digits.  Since $E_n$ is $\mathcal F_{r_n}$-measurable and the
binary block used in $h_{i,n}$ begins after $b_{i,n}$, we have
\[
 \mathbb E(h_{i,n}\mid\mathcal F_{b_{i,n}})=0.
\]
Hence, for every $x\in L_\infty[0,1]$,
\[
 \left|\int_0^1 xh_{i,n}\,d\lambda\right|
 =\left|\int_0^1
   \bigl(x-\mathbb E(x\mid\mathcal F_{b_{i,n}})\bigr)h_{i,n}\,d\lambda\right|
 \le
 \left\|x-\mathbb E(x\mid\mathcal F_{b_{i,n}})\right\|_1.
\]
As $b_{i,n}\to\infty$ and the dyadic conditional expectations converge to
$x$ in $L_1$, the last expression tends to zero.  Thus $h_{i,n}\to0$ weakly.
To pass to $f_i$, fix $x\in L_\infty[0,1]$ and $\varepsilon>0$.  By the uniform
summability choose $N$ so that
$\|x\|_\infty\sum_{n>N}a_n\|h_{i,n}\|_1<\varepsilon$ for every $i$; the
integral against the finite sum over $n\le N$ tends to zero as $i\to\infty$.
Hence $f_i\to0$ weakly in $L_1$.

We may therefore invoke Claim~1 and the subsequent small-perturbation
construction in the proof of \cite[Theorem~1]{JMS2007}.  Thus, for every
$\varepsilon>0$, the Haar basis is $(1+\varepsilon)$-equivalent to a block
basis of every subsequence of $(f_i)$.
Consequently $[f_{i_k}]$ contains a copy of $L_1[0,1]$ for every subsequence
$(f_{i_k})$.
Finally put
\[
 g_i=\frac{f_i}{\|f_i\|_1}.
\tag{20}
\]

\begin{proposition}\label{prop:concrete-JMS}
The sequence $(g_i)$ defined by \emph{(20)} is normalized and weakly null in
$L_1[0,1]$.  For every $\varepsilon>0$, the Haar basis of $L_1[0,1]$ is
$(1+\varepsilon)$-equivalent to a block basis of every subsequence of
$(g_i)$.  Consequently, for every subsequence $(g_{i_k})$, the space
$[g_{i_k}]$ contains an isomorphic copy of $L_1[0,1]$ and is nonreflexive.
\end{proposition}

\begin{proof}
By (19), the normalization factors $\|f_i\|_1^{-1}$ are bounded, so the weak
nullity of $(f_i)$ implies the weak nullity of $(g_i)$, and of course
$\|g_i\|_1=1$.  Moreover $[g_{i_k}]=[f_{i_k}]$ for every subsequence.  The
Haar block-basis conclusion established above for $(f_i)$ passes to $(g_i)$:
a block vector of the $f_i$'s is the same vector after absorbing the factors
$\|f_i\|_1$ into its block coefficients.  The final assertion follows
because the closed span of a sequence equivalent to the Haar basis is
isomorphic to $L_1[0,1]$.
\end{proof}

\section{Explicit realizations for nonatomic measure spaces}\label{sec:Linfty}
Let $(\Omega,\Sigma,\mu)$ be a nonatomic probability space.  The interval
$[0,1]$ plays no essential role in the preceding construction.  Indeed, the
nonatomicity of $\mu$ allows us to reproduce the dyadic structure of
$[0,1]$ inside $\Omega$.

Choose recursively measurable sets
\[
 A_s\in\Sigma,\qquad s\in\{0,1\}^{<\mathbb N},
\]
so that, modulo null sets,
\[
 A_{\varnothing}=\Omega,\qquad
 A_s=A_{s0}\,\dot\cup\,A_{s1},\qquad
 \mu(A_s)=2^{-|s|}.
\tag{21}
\]
Such a family exists because every measurable set of positive measure in a
nonatomic finite measure space can be divided into two measurable subsets of
equal measure.  Let $I_s\subset[0,1]$ be the dyadic interval corresponding to
the binary word $s$.

On dyadic simple functions define
\[
 U\left(\sum_{|s|=m}c_s\mathbf 1_{I_s}\right)
   =\sum_{|s|=m}c_s\mathbf 1_{A_s}.
\tag{22}
\]
The definition is consistent when a function is written at a finer dyadic
level, and
\[
 \left\|U\left(\sum_{|s|=m}c_s\mathbf 1_{I_s}\right)\right\|_{L_1(\mu)}
 =\sum_{|s|=m}|c_s|\,2^{-m}
 =\left\|\sum_{|s|=m}c_s\mathbf 1_{I_s}\right\|_{L_1[0,1]}.
\]
Since dyadic simple functions are dense in $L_1[0,1]$, $U$ extends uniquely
to a linear isometry
\[
 U:L_1[0,1]\longrightarrow L_1(\mu).
\tag{23}
\]
Equivalently, the splitting (21) defines binary digits
$\varepsilon_n(\omega)$ almost everywhere and hence a measurable map
\[
 \pi(\omega)=\sum_{n=1}^{\infty}\frac{\varepsilon_n(\omega)}{2^n}.
\]
For every dyadic interval $I_s$ one has
$\pi^{-1}(I_s)=A_s$ modulo a null set.  The preimage of each dyadic endpoint is null, since it
is contained in a decreasing sequence of sets of measures $2^{-m}$.  Hence
$\mu(\pi^{-1}(I_s))=2^{-|s|}=\lambda(I_s)$.  Since the dyadic intervals
form a $\pi$-system generating the Borel $\sigma$-algebra, it follows that
$\pi_\#\mu=\lambda$.  Therefore $Uf=f\circ\pi$ in $L_1(\mu)$.  Thus the densities used below can be
written concretely as $\widetilde g_i=g_i\circ\pi$.

The space $L_\infty(\mu)$ is a Grothendieck space.  Passing to the completion
of the measure does not change the relevant $L_p$ spaces.  A finite measure
space is localizable, so the complex space $L_\infty(\mu;\mathbb C)$ is a
commutative von Neumann algebra; Pfitzner's \cite[Corollary~7]{Pfitzner1994}
states that von Neumann algebras are Grothendieck.  The corresponding real $L_\infty(\mu)$ is Grothendieck as well: the
underlying real space of a complex Grothendieck space is Grothendieck, and
complemented real subspaces inherit the Grothendieck property.  We also have
the canonical linear isometry
\[
 J_\mu:L_1(\mu)\longrightarrow L_\infty(\mu)^*,\qquad
 (J_\mu f)(x)=\int_\Omega x(\omega)f(\omega)\,d\mu(\omega).
\tag{24}
\]
Indeed, $\|J_\mu f\|\leq\|f\|_1$, while equality follows by taking
$x=\operatorname{sgn}f$.

\begin{theorem}\label{thm:explicit-Linfty}
Let $(g_i)$ be the normalized Johnson--Maurey--Schechtman sequence defined by
\emph{(20)}, put
\[
 \widetilde g_i=Ug_i\in L_1(\mu),
\]
and define
\[
 \boxed{
 y_i(x)=\int_\Omega x(\omega)\widetilde g_i(\omega)\,d\mu(\omega),
 \qquad x\in L_\infty(\mu).}
\tag{25}
\]
Then $(y_i)\subset S(L_\infty(\mu)^*)$ is weakly null, hence normal, and no
subsequence of $(y_i)$ is strongly normal.
\end{theorem}

\begin{proof}
By (23) and (24), $y_i=J_\mu Ug_i$ and $\|y_i\|=1$.  Since $U$ and $J_\mu$
are bounded linear operators and $g_i\to0$ weakly in $L_1[0,1]$, we have
$Ug_i\to0$ weakly in $L_1(\mu)$ and therefore $y_i\to0$ weakly in
$L_\infty(\mu)^*$.  In particular, $(y_i)$ is normal.

Suppose that some subsequence $(y_{i_k})$ were strongly normal.  By
Corollary~\ref{cor:strong-reflexive}, it would have a further subsequence
$(y_{i_{k_l}})$ whose norm-closed linear span is reflexive.  On the other
hand, $J_\mu U$ is an isometry, and hence
\[
 [y_{i_{k_l}}]=(J_\mu U)([g_{i_{k_l}}]).
\]
By Proposition~\ref{prop:concrete-JMS}, $[g_{i_{k_l}}]$ contains a copy of
$L_1[0,1]$ and is nonreflexive.  Its isometric image is therefore
nonreflexive, a contradiction.
\end{proof}

Equivalently, every $y_i$ is represented by the countably additive signed
measure
\[
 d\nu_i=\widetilde g_i\,d\mu.
\tag{26}
\]
Thus no abstract embedding theorem and no ultrafilter are needed.  The only
choice is a dyadic splitting (21) of the nonatomic measure space.  With the
map $\pi$ above, formula (25) may equivalently be written
\[
 y_i(x)=\int_\Omega x(\omega)g_i(\pi(\omega))\,d\mu(\omega).
\]  For
$\Omega=[0,1]$ with Lebesgue measure and the usual dyadic intervals, one may
take $U$ to be the identity; then (25) reduces to the particularly simple
formula
\[
 y_i(x)=\int_0^1x(t)g_i(t)\,dt.
\]

\begin{corollary}\label{cor:sigmafinite}
Let $(\Omega,\Sigma,\mu)$ be a $\sigma$-finite measure space.  Assume that
there exists $A\in\Sigma$ such that $0<\mu(A)<\infty$ and the restricted
measure $\mu|_A$ is nonatomic.  Then $L_\infty(\mu)^*$ contains a normalized
weakly null sequence with no strongly normal subsequence.
\end{corollary}

\begin{proof}
Put $\mu_A=\mu(A)^{-1}\mu|_A$ and apply
Theorem~\ref{thm:explicit-Linfty} to the nonatomic probability space
$(A,\Sigma|_A,\mu_A)$.  The map
\[
 T:L_1(\mu_A)\longrightarrow L_1(\mu),\qquad
 Tf=\frac{1}{\mu(A)}f\mathbf 1_A,
\]
is a linear isometry.  Let $(\widetilde g_i)$ be the normalized sequence in $L_1(\mu_A)$
produced in Theorem~\ref{thm:explicit-Linfty}, and put
$h_i=T\widetilde g_i$.  Then the required functionals are explicitly
\[
 y_i(x)=\int_\Omega x(\omega)h_i(\omega)\,d\mu(\omega),
 \qquad x\in L_\infty(\mu).
\]
The composition $J_\mu T$, where
$J_\mu:L_1(\mu)\to L_\infty(\mu)^*$ is the canonical integration map,
is an isometry.  Since a $\sigma$-finite measure is localizable (after
completion), the same von Neumann algebra argument shows that
$L_\infty(\mu)$ is Grothendieck.  The proof of
Theorem~\ref{thm:explicit-Linfty} therefore applies unchanged.
\end{proof}

\begin{remark}\label{rem:atomic}
The presence of a nonatomic part is what makes this integration model
available.  In a purely atomic probability space, the canonical predual
$L_1(\mu)$ is an $\ell_1$-space and therefore has the Schur property; it
contains no normalized weakly null sequence.  Thus the above construction
cannot be carried out inside the canonical $L_1(\mu)$ part of
$L_\infty(\mu)^*$.  If there are only finitely many atoms, $L_\infty(\mu)$ is
finite-dimensional.  In the infinite purely atomic $\sigma$-finite case,
$L_\infty(\mu)$ is isometric to $\ell_\infty$, and the next section supplies a
different explicit representation, using a free ultrafilter.
\end{remark}

\section{An explicit representation in
\texorpdfstring{$\ell_\infty^*$}{the dual of l-infinity}}\label{sec:explicit}
For $\ell_\infty$ we can also write the counterexample by a concrete formula,
although here one nonconstructive datum remains: the choice of a free
ultrafilter on $\mathbb N$.  We use the same sequence $(g_i)$ from
Section~\ref{sec:concreteJMS}.

Let
\[
 \mathcal D=\{(m,j):m\in\mathbb N,\ 1\le j\le2^m\}
\]
and define the bijection $\rho:\mathcal D\to\mathbb N$ by
\[
 \rho(m,j)=2^m+j-2.
\]
Put
\[
 I_{m,j}=[(j-1)2^{-m},j2^{-m}),
\]
with the harmless convention that $1$ belongs to the last interval.  Fix a
free ultrafilter $\mathcal U$ on $\mathbb N$.  For $f\in L_1[0,1]$ and
$x=(x_r)\in\ell_\infty$, the scalar sequence
\[
 \left(\sum_{j=1}^{2^m}x_{\rho(m,j)}
   \int_{I_{m,j}}f(t)\,dt\right)_{m\ge1}
\]
is bounded by $\|x\|_\infty\|f\|_1$.  Hence its $\mathcal U$-limit exists,
and we define
\[
 (J_{\mathcal U}f)(x)
 =\lim_{m\to\mathcal U}
   \sum_{j=1}^{2^m}x_{\rho(m,j)}
   \int_{I_{m,j}}f(t)\,dt .
\tag{27}
\]
Ultrafilter limits of bounded real scalar sequences respect finite linear
combinations; hence $J_{\mathcal U}$ is linear in $f$ and, for fixed $f$,
linear in $x$.

\begin{lemma}\label{lem:explicit-embedding}
The map
\[
 J_{\mathcal U}:L_1[0,1]\longrightarrow\ell_\infty^*
\]
defined by \emph{(27)} is a linear isometry.
\end{lemma}

\begin{proof}
For $m\in\mathbb N$ set
\[
 \Phi_m(f,x)=\sum_{j=1}^{2^m}x_{\rho(m,j)}
   \int_{I_{m,j}}f(t)\,dt.
\]
Then
\[
 |\Phi_m(f,x)|
 \le \|x\|_\infty\sum_{j=1}^{2^m}
       \left|\int_{I_{m,j}}f\,d\lambda\right|
 \le \|x\|_\infty\|f\|_1.
\]
Hence (27) defines a bounded linear functional and
$\|J_{\mathcal U}f\|\le\|f\|_1$.

For the reverse inequality, define $x^{(f)}\in B_{\ell_\infty}$ by
\[
 x^{(f)}_{\rho(m,j)}=
 \operatorname{sgn}\!\left(\int_{I_{m,j}}f\,d\lambda\right),
\tag{28}
\]
where $\operatorname{sgn}(0)=0$.  Let $\mathcal F_m$ be the dyadic
$\sigma$-algebra of level $m$.  Then
\[
 \sum_{j=1}^{2^m}\left|\int_{I_{m,j}}f\,d\lambda\right|
 =\|\mathbb E(f\mid\mathcal F_m)\|_1.
\]
The dyadic conditional expectations converge to $f$ in $L_1$, and therefore
\[
 \lim_{m\to\infty}
 \sum_{j=1}^{2^m}\left|\int_{I_{m,j}}f\,d\lambda\right|
 =\|f\|_1.
\]
Using (28) in (27) gives
\[
 (J_{\mathcal U}f)(x^{(f)})=\|f\|_1.
\]
Thus $\|J_{\mathcal U}f\|\ge\|f\|_1$, and the proof is complete.
\end{proof}

\medskip
\noindent\textit{Location of the range.}
The image of $J_{\mathcal U}$ lies in the annihilator
\[
        c_0^\perp=\{\phi\in\ell_\infty^*: \phi|_{c_0}=0\}.
\]
Indeed, if $x\in c_0$, then $\rho(m,j)\ge 2^m-1$ for all
$1\le j\le2^m$, and therefore
\[
 |\Phi_m(f,x)|
 \le \|f\|_1\sup_{r\ge2^m-1}|x_r|
 \longrightarrow0.
\]
Hence $(J_{\mathcal U}f)(x)=0$ for every $x\in c_0$.  Thus the explicit
copy of $L_1[0,1]$ constructed above is contained entirely in the singular
part $c_0^\perp$ of $\ell_\infty^*$.

\begin{theorem}\label{thm:explicit-linfty}
Let $(g_i)$ be the normalized sequence defined by \emph{(20)}, and define
\[
 \boxed{
 y_i(x)=\lim_{m\to\mathcal U}
   \sum_{j=1}^{2^m}x_{\rho(m,j)}
   \int_{I_{m,j}}g_i(t)\,dt,
 \qquad x\in\ell_\infty .}
\tag{29}
\]
Then
\[
        (y_i)\subset S(c_0^\perp)\subset S(\ell_\infty^*)
\]
is weakly null, hence normal, and no subsequence of $(y_i)$ is strongly
normal.
\end{theorem}

\begin{proof}
By Lemma~\ref{lem:explicit-embedding}, $y_i=J_{\mathcal U}g_i$ and
$\|y_i\|=1$.  Since $J_{\mathcal U}$ is bounded linear and
$g_i\to0$ weakly in $L_1[0,1]$, we have $y_i\to0$ weakly in
$\ell_\infty^*$; in particular $(y_i)$ is normal.

Suppose that $(y_{i_k})$ were strongly normal.  By
Corollary~\ref{cor:strong-reflexive}, it would have a further subsequence
$(y_{i_{k_l}})$ whose norm-closed linear span is reflexive.  On the other
hand, $J_{\mathcal U}$ is an isometry, so
\[
 [y_{i_{k_l}}]=J_{\mathcal U}([g_{i_{k_l}}]).
\]
Proposition~\ref{prop:concrete-JMS} implies that
$[g_{i_{k_l}}]$ contains a copy of $L_1[0,1]$ and is therefore
nonreflexive.  This contradiction proves the assertion.
\end{proof}

Equivalently, viewing $\ell_\infty^*$ as $ba(\mathbb N)$ and writing
$y_i(A):=y_i(\mathbf 1_A)$, the finitely additive measure corresponding to
$y_i$ is given, for $A\subset\mathbb N$, by
\[
 y_i(A)=\lim_{m\to\mathcal U}
 \sum_{\substack{1\le j\le2^m\\ \rho(m,j)\in A}}
 \int_{I_{m,j}}g_i(t)\,dt.
\tag{30}
\]
Thus the counterexample in $\ell_\infty^*$ is not merely obtained from an
abstract embedding theorem: after fixing $\mathcal U$ and the recursive
dyadic data above, every term of the sequence is given by the concrete
formula (29), or equivalently (30), and the entire sequence lies in
$c_0^\perp$.  The dependence on a free ultrafilter is the only noncanonical
part of this representation.

\section{Obstruction to strongly normal block sequences}\label{sec:block-obstruction}

The counterexamples considered above exclude strongly normal subsequences.  It
is natural to ask whether passing from subsequences to block sequences can
restore strong normality.  The answer is again negative, and in a rather
strong form.

Recall that if $(y_n)$ is a basic sequence, a \emph{normalized block sequence}
of $(y_n)$ is a sequence
\[
 z_j=\sum_{n=p_j}^{q_j}a_n^{(j)}y_n,
 \qquad
 p_j\le q_j<p_{j+1},
 \qquad \|z_j\|=1.
\tag{31}
\]
We shall use a construction of Popov.  A Banach space is called
\emph{hereditarily $\ell_1$} if every infinite-dimensional closed subspace
contains a further subspace isomorphic to $\ell_1$.

\begin{theorem}[Popov \cite{Popov2005}]\label{thm:Popov}
There exists a closed hereditarily $\ell_1$ subspace
\[
 P\subset L_1[0,1]
\]
which does not have the Schur property.
\end{theorem}

The absence of the Schur property is precisely what permits weakly null
sequences which are not norm null, while hereditary $\ell_1$ prevents any
infinite-dimensional subspace from being reflexive.  These two features fit
the criterion from Section~\ref{sec:criteria} exactly.

\begin{lemma}\label{lem:Popov-basic}
The space $P$ from Theorem~\ref{thm:Popov} contains a normalized weakly null
basic sequence $(x_n)$ such that every normalized block sequence $(u_j)$ of
$(x_n)$ satisfies
\[
 [u_j]\text{ contains a copy of }\ell_1,
\tag{32}
\]
and hence $[u_j]$ is nonreflexive.  The same conclusion holds for every
subsequence of every such block sequence.
\end{lemma}

\begin{proof}
Since $P$ fails the Schur property, there is a weakly convergent sequence in
$P$ which is not norm convergent.  Subtracting its weak limit, we obtain a
weakly null sequence $(v_n)$ which is not norm null.

Passing to a subsequence, there is a $\delta>0$
with
\[
 \|v_n\|\ge\delta\qquad(n\in\mathbb N).
\]
Weakly convergent sequences are norm bounded, so after putting
\[
 w_n=\frac{v_n}{\|v_n\|}
\]
we obtain a normalized weakly null sequence.

By the
Bessaga--Pe\l czy\'nski selection principle
\cite[Chapter~V]{Diestel1984}, $(w_n)$ has a basic subsequence.  Relabel this
subsequence as $(x_n)$.

Let $(u_j)$ be any normalized block sequence of $(x_n)$.  Every block sequence
of a basic sequence is basic, so $[u_j]$ is an infinite-dimensional closed
subspace of $P$.  Since $P$ is hereditarily $\ell_1$, the space $[u_j]$
contains a copy of $\ell_1$ and therefore is not reflexive.

A subsequence of
$(u_j)$ is again a normalized block sequence of $(x_n)$, so the same argument
applies to each of its subsequences.
\end{proof}

We can now strengthen the main negative theorem.

\begin{theorem}\label{thm:no-strong-blocks}
Let $E$ be a nonreflexive Grothendieck Banach space.  Then $E^*$ contains a
normalized weakly null basic sequence $(y_n)$ such that no normalized block
sequence of $(y_n)$ is strongly normal.
\end{theorem}

\begin{proof}
Let $(x_n)\subset P\subset L_1[0,1]$ be the sequence from
Lemma~\ref{lem:Popov-basic}.  By Lemma~\ref{lem:L1insideHaydon}, there is a
linear isometry
\[
 J:L_1[0,1]\longrightarrow E^*.
\]

Put
\[
 y_n=Jx_n,\qquad n\in\mathbb N.
\tag{33}
\]
Since $J$ is an isometry, $(y_n)$ is normalized and basic.  Since $(x_n)$ is
weakly null in $L_1[0,1]$ and $J$ is bounded linear, $(y_n)$ is weakly null in
$E^*$ and hence normal.

Suppose, towards a contradiction, that a normalized block sequence $(z_j)$ of
$(y_n)$ is strongly normal.  Write
\[
 z_j=\sum_{n=p_j}^{q_j}a_n^{(j)}y_n
     =J\!\left(\sum_{n=p_j}^{q_j}a_n^{(j)}x_n\right)
     =Ju_j .
\tag{34}
\]
Because $J$ is an isometry, $(u_j)$ is a normalized block sequence of
$(x_n)$.

By Corollary~\ref{cor:strong-reflexive}, the strongly normal
sequence $(z_j)$ has a subsequence $(z_{j_k})$ whose norm-closed linear span is
reflexive.

On the other hand,
\[
 [z_{j_k}]=J([u_{j_k}]).
\tag{35}
\]

By Lemma~\ref{lem:Popov-basic}, $[u_{j_k}]$ contains a copy of $\ell_1$ and is
nonreflexive.  Its isometric image in (35) is therefore nonreflexive, a
contradiction.
\end{proof}

Thus the obstruction can persist under arbitrary blocking, not merely under
passage to subsequences.  In fact, the proof gives the following useful
structural criterion.

\begin{proposition}\label{prop:no-block-reflexive}
Let $E$ be a Grothendieck Banach space and let $(y_n)\subset S(E^*)$ be a
normal basic sequence.  If $[y_n]$ contains no infinite-dimensional reflexive
subspace, then no normalized block sequence of $(y_n)$ is strongly normal.
\end{proposition}

\begin{proof}
If a normalized block sequence $(z_j)$ were strongly normal,
Corollary~\ref{cor:strong-reflexive} would give a subsequence $(z_{j_k})$ such
that $[z_{j_k}]$ is reflexive.  This is an infinite-dimensional reflexive
subspace of $[y_n]$, contrary to the hypothesis.
\end{proof}

For Popov's space the hypothesis of Proposition~\ref{prop:no-block-reflexive}
is satisfied in the strongest possible way: every infinite-dimensional
closed subspace contains $\ell_1$.

There is also a direct realization in the canonical $L_1$ part of the dual of
$L_\infty[0,1]$.

\begin{corollary}\label{cor:Linfty-no-strong-blocks}
There exists a normalized weakly null basic sequence $(y_n)$ in
$(L_\infty[0,1])^*$ such that no normalized block sequence of $(y_n)$ is
strongly normal.  It can be chosen in the canonical copy of $L_1[0,1]$ and is
given by
\[
 y_n(f)=\int_0^1 f(t)x_n(t)\,dt,
 \qquad f\in L_\infty[0,1],
\tag{36}
\]
where $(x_n)$ is the sequence from Lemma~\ref{lem:Popov-basic}.
\end{corollary}

\begin{proof}
Apply the canonical isometry
\[
 J_\lambda:L_1[0,1]\longrightarrow (L_\infty[0,1])^*,
 \qquad
 (J_\lambda x)(f)=\int_0^1fx\,d\lambda,
\]
to the sequence $(x_n)$.  The proof of Theorem~\ref{thm:no-strong-blocks}
then applies verbatim, using that $L_\infty[0,1]$ is Grothendieck.
\end{proof}

Theorem~\ref{thm:no-strong-blocks} shows that replacing subsequences by block
sequences does not rescue Problem~(P), even within the class of nonreflexive
Grothendieck spaces.  The Johnson--Maurey--Schechtman construction used in the
main theorem guarantees nonreflexivity of the closed span of every
\emph{subsequence}; Popov's hereditary $\ell_1$ construction provides the
stronger hereditary obstruction needed to control every normalized
\emph{block sequence}.

\section{A hereditarily \texorpdfstring{$\ell_1$}{l1} subspace containing no strongly normal sequence}\label{sec:hereditary-strongest}

We first separate the part of the argument which has nothing to do with
Grothendieck spaces.

\begin{proposition}\label{prop:L1-hereditary-basis}
Let $X$ be a Banach space which contains an isomorphic copy of $L_1[0,1]$.
Then $X$ contains a closed subspace $Y$ with the following properties:
\begin{enumerate}
\item $Y$ is hereditarily $\ell_1$;
\item $Y$ has a normalized weakly null Schauder basis.
\end{enumerate}
In particular, $Y$ fails the Schur property and contains no
infinite-dimensional reflexive subspace.
\end{proposition}

\begin{proof}
Let $T:L_1[0,1]\to X$ be an isomorphic embedding.  Thus $T$ is one-to-one,
$T(L_1)$ is closed, and there are constants $c,C>0$ such that
\[
        c\|f\|_1\le\|Tf\|\le C\|f\|_1
        \qquad(f\in L_1[0,1]).
\]

By Lemma~\ref{lem:Popov-basic}, Popov's space $P$ contains a normalized weakly
null basic sequence $(x_n)$.  Put $X_0=[x_n]$.  Since $X_0$ is an
infinite-dimensional closed subspace of the hereditarily $\ell_1$ space $P$,
$X_0$ is itself hereditarily $\ell_1$: every infinite-dimensional closed
subspace of $X_0$ is also one of $P$.

Define
\[
        v_n=Tx_n,
        \qquad
        y_n=\frac{v_n}{\|v_n\|}.
\]
Since $\|x_n\|=1$, the embedding inequalities give
\[
        c\le\|v_n\|\le C.
\]
Thus the normalization factors $\|v_n\|^{-1}$ are bounded above and away from
zero.

We verify weak nullity directly.  For $x^*\in X^*$,
\[
 x^*(y_n)
 =\frac{(T^*x^*)(x_n)}{\|Tx_n\|}.
\]
The numerator tends to zero because $x_n\to0$ weakly in $L_1[0,1]$, while the
denominator is at least $c$.  Hence $x^*(y_n)\to0$ for every $x^*$, so
$y_n\to0$ weakly.

The sequence $(Tx_n)$ is basic because $T$ restricts to an isomorphism from
$X_0$ onto $T(X_0)$.  Multiplying the $n$-th vector of a basic sequence by a
nonzero scalar preserves basicity provided the scalars and their reciprocals
are bounded; this holds here.  Therefore $(y_n)$ is a normalized weakly null
basic sequence.

Scalar rescaling does not change the closed linear span, so
\[
        [y_n]=[Tx_n]=T(X_0).
\]
The restriction $T|_{X_0}:X_0\to T(X_0)$ is an isomorphism.  The property
``hereditarily $\ell_1$'' is invariant under isomorphism: an
infinite-dimensional closed subspace of $T(X_0)$ pulls back to one of $X_0$,
which contains a copy of $\ell_1$, and applying $T$ transports that copy back.
Hence $Y:=T(X_0)$ is hereditarily $\ell_1$.

Finally, $Y$ contains no infinite-dimensional reflexive subspace.  If
$Z\subset Y$ were such a subspace, hereditary $\ell_1$ would force $Z$ to
contain a closed subspace isomorphic to $\ell_1$, but every closed subspace of
a reflexive space is reflexive and $\ell_1$ is not.  The weakly null
normalized basis also shows directly that $Y$ fails the Schur property.
\end{proof}

We now apply this observation to duals of nonreflexive Grothendieck spaces.

\begin{theorem}\label{thm:hereditary-no-strong}
Let $E$ be a nonreflexive Grothendieck Banach space.  Then there exists a
closed subspace $Y\subset E^*$ such that
\begin{enumerate}
\item $Y$ is hereditarily $\ell_1$;
\item $Y$ has a normalized weakly null Schauder basis $(y_n)$;
\item no sequence $(z_n)\subset S(Y)$ is strongly normal in $E^*$.
\end{enumerate}
Consequently, $(y_n)$ is a normal basic sequence, while no normalized
subsequence, no normalized block sequence, and in fact no normalized sequence
contained in $Y=[y_n]$ is strongly normal.
\end{theorem}

\begin{proof}
By Haydon's theorem and Lemma~\ref{lem:L1insideHaydon}, the dual $E^*$ contains
an isometric copy of $L_1[0,1]$.  Apply
Proposition~\ref{prop:L1-hereditary-basis} with $X=E^*$.  We obtain a closed
subspace $Y\subset E^*$ which is hereditarily $\ell_1$ and has a normalized
weakly null Schauder basis $(y_n)$.  This proves assertions (1) and (2).
Because weak convergence in $E^*$ implies weak-star convergence,
$(y_n)$ is also a normal sequence.

It remains to prove the much stronger assertion (3).  Let $(z_n)$ be
\emph{any} sequence in the unit sphere $S(Y)$; it need not be a subsequence or
a block sequence of $(y_n)$.  Suppose, for contradiction, that $(z_n)$ is
strongly normal in $E^*$.

Since $E$ is Grothendieck, Corollary~\ref{cor:strong-reflexive} supplies a
subsequence $(z_{n_k})$ such that
\[
        Z=[z_{n_k}]
\]
is reflexive.  The selected subsequence is weak-star basic (this is how the
corollary is obtained), so its vectors are linearly independent.  Therefore
$Z$ is infinite-dimensional.

But every $z_{n_k}$ belongs to $Y$, and $Y$ is closed, so
$Z\subset Y$.  Since $Y$ is hereditarily $\ell_1$, the infinite-dimensional
closed subspace $Z$ contains an isomorphic copy of $\ell_1$.  This is
impossible in a reflexive space.  The contradiction proves that no normalized
sequence contained in $Y$ is strongly normal.

The final sentence follows immediately: subsequences and normalized block
sequences of the basis $(y_n)$ are particular normalized sequences in $Y$.
\end{proof}

Theorem~\ref{thm:hereditary-no-strong} is stronger than the
subsequence and block obstructions established earlier: it is a property of
an entire closed subspace of $E^*$, not merely of a chosen basis or of its
block structure.  In particular, the normalized weakly null basis supplied
there has no strongly normal normalized subsequence or block sequence, and
indeed its whole closed linear span contains no strongly normal normalized
sequence.

\section{Concluding remarks}
The results above give three complementary descriptions of the same
phenomenon.  First, inside the class of Grothendieck Banach spaces,
reflexivity is characterized by the requirement that every normal sequence
in the dual admit a strongly normal subsequence.  Second, for an arbitrary
Banach space, Theorem~\ref{thm:criteria-general} identifies the local
mechanism exactly: a normal sequence has a strongly normal subsequence if and
only if some subsequence detects a separable quotient.  For Grothendieck
spaces this is equivalent to the existence of a subsequence with reflexive
norm-closed linear span.  Third, Theorem~\ref{thm:hereditary-no-strong}
shows that, in every nonreflexive Grothendieck space, the obstruction may be
concentrated in a closed hereditarily $\ell_1$ subspace of the dual, where no
normalized sequence is strongly normal.

The counterexamples of K\k{a}kol~\cite{Kakol2026} arise from a different,
combinatorial mechanism and contain refined information involving cardinal
invariants.  The interaction between those cardinal phenomena and the
quotient criterion of Section~\ref{sec:criteria} suggests a separate line of
investigation, which is not pursued here.

\section*{Declaration of generative AI and AI-assisted technologies in the manuscript preparation process}
During the preparation of this work, the author used OpenAI's ChatGPT to assist
with exploratory mathematical discussion, manuscript organization, language
refinement, consistency checks, and exposition.  After using this tool, the
author independently reviewed and verified all mathematical statements, proofs,
references, and conclusions, edited the content as needed, and takes full
responsibility for the content of the article.

\normalsize

\end{document}